\documentclass[11pt]{article}

\usepackage[a4paper, hmargin={2.8cm, 2.8cm}, vmargin={2.5cm, 2.5cm}]{geometry}  

\usepackage[usenames]{color}
\usepackage[english]{babel}
\usepackage[utf8]{inputenc}
\usepackage{amsmath,amssymb,amsfonts,amsthm,amscd} 
\usepackage[numbers]{natbib}
\usepackage[mathscr]{euscript}
\usepackage{nicefrac}         
\usepackage{enumerate,graphicx,nicefrac}
\usepackage{paralist}
\usepackage{fancyhdr}
\usepackage{wrapfig}
\usepackage[all]{xy}
\usepackage{braket}
\usepackage{titlesec}
\usepackage{mathrsfs}
\usepackage{tikz,tikz-cd}
\usetikzlibrary{arrows}
\usepackage[page]{appendix}
\usepackage{comment}
\usepackage[nottoc,numbib]{tocbibind}
\usepackage{slashed}
\titlespacing*{\chapter}{0pt}{0pt}{40pt}
\usepackage{hyperref}
\usepackage{multirow} 

\usepackage{ifpdf}
\ifpdf
	\usepackage{epstopdf}
\fi

\theoremstyle{plain}			
	\newtheorem{theorem}{Theorem}
	\newtheorem{corollary}[theorem]{Corollary}
	\newtheorem{lemma}[theorem]{Lemma}
	\newtheorem{proposition}[theorem]{Proposition}

\theoremstyle{definition}		
	\newtheorem{definition}[theorem]{Definition}
	
	\newtheorem{example}[theorem]{Example}

\numberwithin{theorem}{section} 

\theoremstyle{plain}			
	\newtheorem*{theorem*}{Theorem}
	\newtheorem*{corollary*}{Corollary}
	\newtheorem*{lemma*}{Lemma}
	\newtheorem*{definition*}{Definition}
	\newtheorem*{proposition*}{Proposition}
	
\theoremstyle{definition}		
	\newtheorem*{problem*}{Problem}
	\newtheorem*{example*}{Example}
	\newtheorem*{remark*}{Remark}

\newcommand{\m}[1]{\mathbb{#1}}
\newcommand{\NN}{{\m{N}}}

\newcommand{\ZZ}{{\m{Z}}}

\newcommand{\CC}{{\m{C}}}

\newcommand{\defeq}{\mathrel{\mathop:}=}

\providecommand{\norm}[1]{\left\Vert#1\right\Vert}

\begin{document}
\title{Faithful tracial states on quotients of $C^{\ast}$-algebras}

\author{Henning Olai Milhøj}

\maketitle 

\begin{abstract} 
\noindent We study the existence of faithful tracial states on $C^{\ast}$-algebras as well as the stronger proerty that all quotients admit faithful tracial states. We provide a sufficient and necessary condition for when $C^{\ast}$-algebras admit faithful tracial states in terms of the Cuntz semigroup and use this to give an equivalent formulation for all quotients to admit faithful tracial states. We relate this to the notion of strong quasidiagonality, and show that any amenable discrete group with faithful tracial states on all quotients of the corresponding group-$C^{\ast}$-algebra is strongly quasidiagonal under the condition that all quotients satisfy the Universal Coefficient Theorem.
\end{abstract}

\section{Introduction}
A \emph{tracial state} on a $C^{\ast}$-algebra $A$ is a positive linear functional $\tau$ on $A$ with $\norm{\tau}=1$ and $\tau(ab)=\tau(ba)$ for all $a,b\in A$. Tracial states show up in many different areas of $C^{\ast}$-algebras, and while in many cases it is obvious whether or not a given $C^{\ast}$-algebra admits a tracial state or not, it is still a source of current research to decide whether some specific $C^{\ast}$-algebras admit tracial states, for example within the field of quantum information theory \citep[Section 5]{Paulsen-QIT-2016}. Combining the works of Blackadar-Handelman and Haagerup, we know that any unital, stably finite and exact $C^{\ast}$-algebra necessarily admits a tracial state \citep{Blackadar-Handelman-1982, Haagerup-Quasitraces-2014}. It is fairly easy via a Hahn-Banach argument to see that a unital $C^{\ast}$-algebra fails to admit a tracial state if and only if the norm-closure of the span of additive commutators is dense. This condition has been improved by Haagerup \citep[Lemma 2.1]{Haagerup-Quasitraces-2014} and Pop \citep[Theorem 1]{Pop-2001}. A tracial state is said to be \emph{faithful} if $\tau(a^{\ast}a)=0$ if and only if $a=0$. Examples of $C^{\ast}$-algebras with faithful tracial states include reduced $C^{\ast}$-algebras of discrete groups, separable residually finite-dimensional $C^{\ast}$-algebras, and, by Blackadar-Handelman-Haagerup, simple unital stably finite exact $C^{\ast}$-algebras. While there are conditions ensuring the existence of faithful tracial states, see e.g.\ Proposition \ref{prop:cuntz-pedersen} below, these are not always easy to use. There are interesting classes of $C^{\ast}$-algebras including full group $C^{\ast}$-algebras and maximal tensor products of $C^{\ast}$-algebras, where it is difficult to decide if they admit faithful tracial states, and where the existence or non-existence of faithful tracial states has profound implications, see e.g.\ Proposition \ref{prop:max-min-faithful-tracial-state} and the subsequent discussion. \\ \\
We begin our paper in Section 2 by reviewing mostly known results regarding existence of separating families of tracial states. In Section 3, we examine when all quotients of a $C^{\ast}$-algebras admit faithful tracial states, which we coin the QFTS property. We prove that, under certain conditions, this property is equivalent to having no stable intermediate quotients. This can be seen as a converse to the well-known result that admitting stable ideals (or, in general, stable $C^{\ast}$-subalgebras) is an obstruction to admitting a faithful tracial state. Lastly, in Section 4, we extend a result due to Li \citep{Li-2019-Arxiv} that quasidiagonality and strong quasidiagonality coincide for the class of just-infinite $C^{\ast}$-algebras introduced by Grigorchuk-Musat-Rørdam \citep{Grigorchuk-Musat-Rordam-2018}.
\section{Separating families of tracial states on $C^{\ast}$-algebras}
This section is an overview of the mostly known results regarding separating families of tracial states. Let $A$ be a unital $C^{\ast}$-algebra and denote by $T(A)$ the Choquet simplex of tracial states on $A$.  For a tracial state $\tau\in T(A)$, we denote by $I_{\tau}$ the \emph{trace-kernel} ideal consisting of all $a\in A$ with $\tau(a^{\ast}a)=0$. Note that $I_{\tau}$ is a closed two-sided ideal in $A$, and that $\tau$ is faithful if and only if $I_{\tau}=\{0\}$. We say that $A$ has a \emph{separating family of tracial states} if, for any positive non-zero $a\in A$, there exists $\tau\in T(A)$ such that $\tau(a)\neq 0$ or, equivalently, if $\bigcap_{\tau \in T(A)}I_{\tau}=\{0\}$. In general, admitting a faithful tracial state clearly implies having a separating family of tracial states, and the two properties are equivalent for separable, unital $C^{\ast}$-algebras:
\begin{proposition} \label{prop:faithful-trace-iff-separating-family}
Let $A$ be a separable, unital $C^{\ast}$-algebra. The following are equivalent.
\begin{enumerate}
\item $A$ admits a faithful tracial state,
\item $A$ has a separating family of tracial states,
\item each non-zero ideal in $A$ admits a non-zero bounded positive trace.
\end{enumerate}
\end{proposition}
\begin{proof}
The implications (i)$\Rightarrow$(ii)$\Rightarrow$(iii) are trivial. In order to prove (ii)$\Rightarrow$(i), we first show that, for any non-zero $a\in A_+$, there exists a tracial state $\tau\in T(A)$ such that $\norm{a+I_{\tau}}>\norm{a}/2$. We may, without loss of generality, assume that $\norm{a}=1$. Let $g\colon [0,1]\to [0,1]$ be the piecewise linear continuous function which is zero on $[0,1/2]$ and such that $g(1)=1$. By the continuous functional calculus, $g(a)$ is a non-zero, positive element in $A$ and, by assumption, there thus exists a tracial state $\tau\in T(A)$ such that $\tau(g(a))\neq 0$. It then follows that $g(a+I_{\tau})=g(a)+I_{\tau}\neq 0$, and as $g$ is zero on $[0,1/2]$, we hence conclude that $\norm{a+I_{\tau}}>1/2$. We shall now use this fact in showing that $A$ admits a faithful tracial state. By separability, there exists a countable norm-dense sequence $(a_n)_{n\in \NN}$ of positive elements in $A$. By the above, we can, for each $n\in \NN$, find a tracial state $\tau_n$ for which $\norm{a_n+I_{\tau_n}}>\norm{a_n}/2$. Let $\tau=\sum_{n\in \NN}\tau_n / 2^n$. It is easily verified that $\tau$ is a tracial state, and that $I_{\tau} = \bigcap_{n\in \NN}I_{\tau_n}$. By this construction, it follows that, for any $n\geq 1$, 
\begin{align*}
\norm{a_n+I_{\tau}}\geq \norm{a_n+I_{\tau_n}}\geq \frac{\norm{a_n}}{2},
\end{align*}
and hence by continuity $\norm{a+I_{\tau}}\geq \norm{a}/2$ for all positive $a\in A$. This implies that $\tau(a)\neq 0$ for all positive non-zero $a\in A$, i.e., $\tau$ is faithful. \\ \\
We now prove the remaining implication (iii)$\Rightarrow$(ii). Consider the ideal $I_0 = \bigcap_{\tau \in T(A)}I_{\tau}$ and observe that this is non-zero if and only if (ii) is false. For the sake of reaching a contradiction, suppose that $I_0$ is non-zero. By assumption, there exists a non-zero bounded positive trace on $I_0$, which can be extended to all of $A$ by \citep[Theorem 3.3.9]{Murphy-1990}. Upon normalising to ensure unitality of this trace, we can assume that it is a tracial state; call it $\tau_0$. But since $\tau_0$ is non-zero on $I_0$, we reach the contradiction that $I_{0}$ is not contained in $I_{\tau_0}$.
\end{proof}
There is a purely algebraic reformulation of admitting a separating family of tracial states due to Cuntz-Pedersen \citep[Theorem 3.4]{Cuntz-Pedersen-1979}.
\begin{proposition}[Cuntz-Pedersen, 1979] \label{prop:cuntz-pedersen}
A $C^{\ast}$-algebra $A$ has a separating family of tracial states if and only if, for any $x_1,\ldots, x_n\in A$, if $\sum_{i} x_{i}x_{i}^{\ast}\leq \sum_{i} x_{i}^{\ast} x_{i}$, then $\sum_{i} x_i x_{i}^{\ast} = \sum_{i} x_{i}^{\ast}x_{i}$.
\end{proposition}
As mentioned in the introduction, it can be difficult to verify if a given $C^{\ast}$-algebra has a separating family of tracial states or not. However, there are certain classes where it is well-known and easy to check. 
\begin{example} \label{ex:separating-families-examples}
The following classes of $C^{\ast}$-algebras all admit separating families of tracial states.
\begin{enumerate}
\item Finite-dimensional $C^{\ast}$-algebras,
\item Abelian $C^{\ast}$-algebras,
\item Residually finite dimensional $C^{\ast}$-algebras,
\item $C^{\ast}_{r}(G)$ for any discrete group $G$,
\item Simple, unital, stably finite, exact $C^{\ast}$-algebras.
\end{enumerate}
\end{example}
\begin{proof}
The statements (i) and (ii) are trivial. \\ \\
For (iii), note that the separating finite-dimensional *-representation composed with any separating family of tracial states on the finite-dimensional quotients form a separating family of tracial states. \\ \\
(iv): The canonical tracial state on $C^{\ast}_{r}(G)$ is always faithful whenever $G$ is a discrete group. \\ \\
(v): Combine \citep{Haagerup-Quasitraces-2014} and \citep{Blackadar-Handelman-1982} to obtain a tracial state, which is necessarily faithful by simplicity.
\end{proof}
Note that while the reduced group $C^{\ast}$-algebras of discrete groups always admit faithful tracial states, the same does not hold for full group-$C^{\ast}$-algebras. Bekka showed \citep[Corollary 5]{Bekka-2007} that the $C^{\ast}$-algebras $C^{\ast}(\mathrm{SL}_{n}(\ZZ))$ do not admit faithful tracial states for any $n\geq 3$. \\ \\
We now review some permanence properties and lack thereof. For part (v) below, recall that if $I$ is an ideal in a $C^{\ast}$-algebra $A$ and if $I$ has a bounded tracial state $\tau$, then there exists a unique tracial state $\tau'$ on $A$ extending $\tau$; moreover, if $(e_{i})_{i\in I}$ is an approximate unit for $I$, then $\tau'$ is given by
\begin{align*}
\tau'(a) = \lim_{i}\tau(ae_{i}), \qquad a\in A.
\end{align*}
We shall call this tracial extension the \emph{canonical} extension of $\tau$ to $A$.
\begin{proposition} \label{prop:permanence-properties-sep-family}
Let $A$ and $B$ be unital $C^{\ast}$-algebras.
\begin{enumerate}
\item If $A$ has a separating family of tracial states, then so does any $C^{\ast}$-subalgebra of $A$.
\item The minimal tensor product $A\otimes_{\mathrm{min}}B$ has a separating family of tracial states state if and only if both $A$ and $B$ do.
\item $A$ has a separating family of tracial states if and only if $M_n(A)$ does for some (hence all) $n\in \NN$.
\item If $A=A_1\oplus \ldots \oplus A_n$ is a finite direct sum of $C^{\ast}$-algebras each with separating families of tracial states, then $A$ has a separating family of tracial states.
\item If $I$ is an ideal in $A$ and $I$ contains a separating family of tracial states, then their canonical tracial extensions to $A$ form a separating family of tracial states if and only if $I$ is an essential ideal in $A$. In particular, if any non-unital $C^{\ast}$-algebra $I$ has a separating family of tracial states, then so does the multiplier algebra $\mathcal{M}(I)$.
\item If $I$ is an ideal in $A$ with a separating family of tracial states, and if $A/I$ has a separating family of tracial states, then so does $A$. In other words, admitting separating families of tracial states is preserved by taking extensions.
\end{enumerate} 
\end{proposition}
\begin{proof}
(i): This is immediate. \\ \\
(ii): Since $A$ and $B$ can be realised as $C^{\ast}$-subalgebras of $A\otimes_{\mathrm{min}} B$, one direction is immediate by the use of (i). So assume that both $A$ and $B$ have a separating family of tracial states. Suppose, for the sake of reaching a contradiction, that the ideal
\begin{align*}
I=\bigcap_{\tau \in T(A\otimes_{\mathrm{min}}B)}I_{\tau}
\end{align*}
is non-zero. By Kirchberg's slice lemma, see \citep[Lemma 4.1.9]{Rordam-2002}, we find a non-zero element $z\in A\otimes_{\mathrm{min}}B$ such that $zz^{\ast}=a\otimes b$ for some positive $a\in A$ and $b\in B$ and such that $z^{\ast}z\in I$. Observe that this implies $z\in I$. As $A$ and $B$ both have separating family of tracial states, we may find tracial states $\tau_A$ and $\tau_B$ on $A$ and $B$, respectively, such that $\tau_A(a)\neq 0$ and $\tau_B(b)\neq 0$. But then we clearly reach a contradiction, as
\begin{align*}
(\tau_A\otimes \tau_B)(z^{\ast}z) &= (\tau_A\otimes \tau_B)(zz^{\ast}) = (\tau_A\otimes \tau_B)(a\otimes b) \\
&= (\tau_A\otimes \tau_B)(a\otimes b) = \tau_A(a)\tau_B(b)\neq 0,
\end{align*}
which would imply that $z\not \in I_{\tau}$, contradicting the construction of $z$. \\ \\
(iii): This follows from (ii), since $M_n(A)\cong M_n(\CC)\otimes_{\mathrm{min}} A$ and $M_n(\CC)$ admits a faithful tracial state for every $n\in \NN$. \\ \\
(iv) is obvious. \\ \\
(v): Suppose that $I$ is an essential ideal in $A$. If $\tau\in T(I)$ is a tracial state, we denote by $\tau'$ the canonically extended tracial state on $A$. Assume that $a\in A$ satisfies that $\tau'(a^{\ast}a)=0$ for all $\tau'\in T(A)$. In particular, for any $b\in I$, we find that
\begin{align*}
\tau((ba)^{\ast}(ba)) \leq \norm{b^{\ast}b}\tau'(a^{\ast}a) = 0,
\end{align*}
for all $\tau\in T(I)$. Since $I$ has a separating family of tracial states, this means that $ba=0$ for all $b\in I$ or, equivalently, that $Ia=0$. $I$ being an essential ideal in $A$ then implies that $a=0$ as desired. Conversely, suppose that the canonical extensions of tracial states of $I$ on $A$ are separating. Let $a\in A^{+}\setminus \{0\}$ be arbitrary and suppose that $Ia=aI=0$. By assumption, there exists a tracial state $\tau$ on $I$ for which the canonical extension $\tau'$ on $A$ satisfies that $\tau'(a)\neq 0$. However, this contradicts the assumption that $a$ is orthogonal to $I$, since
\begin{align*}
\tau'(a) = \lim_{i}\tau(ae_{i}) = 0.
\end{align*}
(vi): Let $a\in A^{+}\setminus \{0\}$ be arbitrary and let $\pi\colon A\to A/I$ denote the quotient map. If $a\in I$, then there exists a tracial state $\tau$ on $I$ such that $\tau(a)>0$. On the other hand, if $a\not \in I$, then $\pi(a)\in (A/I)^{+}\setminus \{0\}$ and, by assumption, there exists a tracial state $\tau'$ on $A/I$ for which $\tau'(\pi(a))>0$. This proves that $A$ has a separating family of tracial states.
\end{proof}
The following proposition is an easy extension of the case with faithful tracial states, and the proofs are identical, so we skip the proof and refer to \citep{Murphy-1990}.
\begin{proposition} \label{prop:faithful-trace-implies-stably-finite-and-no-stable}
Let $A$ be a $C^{\ast}$-algebra with a separating family of tracial states. Then (the unitisation of) $M_n(A)$ contains no proper isometries for any $n\in \NN$, and $A$ has no stable $C^{\ast}$-subalgebra.
\end{proposition}
We now give some counterexamples to other possible permanence properties.
\begin{example}
Having a separating family of tracial states does not pass to inductive limits.
\end{example}
\begin{proof}
Consider the unitization $\tilde{\mathbb{K}}$ of the compact operators. This is a unital AF-algebra, hence it is the inductive limit of finite-dimensional $C^{\ast}$-algebras admitting faithful tracial states, but it does not admit a faithful tracial state itself, since it contains $\mathbb{K}$ as an ideal. 
\end{proof}
\begin{example} \label{ex:faithful-tracial-not-passing-to-quotient}
Having a separating family of tracial state does not pass to quotients.
\end{example}
\begin{proof}
Any separable $C^{\ast}$-algebra can be realised as a quotient of the residually finite-dimensional $C^{\ast}$-algebra $C^{\ast}(\mathbb{F}_{\infty})$.
\end{proof}
It does not hold in general that admitting separating families of tracial states passes to \emph{maximal} tensor products. By following the proof of \citep[Proposition 3.13]{Kirchberg-Rordam-2015} in the case $D=A\otimes_{\mathrm{max}}B$, one obtains the next proposition.
\begin{proposition} \label{prop:max-min-faithful-tracial-state}
Let $A$ and $B$ be unital $C^{\ast}$-algebras. If $A\otimes_{\mathrm{max}}B$ has a separating family of tracial states, then $A\otimes_{\mathrm{max}} B = A\otimes_{\mathrm{min}} B$.
\end{proposition}
We can use this proposition to show that Proposition \ref{prop:permanence-properties-sep-family}(ii) fails for \emph{maximal} tensor products. It is well-known, see e.g. \citep{Takesaki-1964}, that $C^{\ast}_{r}(\mathbb{F}_2)\otimes_{\mathrm{max}}C^{\ast}_{r}(\mathbb{F}_2) \neq C^{\ast}_{r}(\mathbb{F}_2)\otimes_{\mathrm{min}}C^{\ast}_{r}(\mathbb{F}_2)$. By Proposition \ref{prop:max-min-faithful-tracial-state}, the maximal tensor product does not admit a separating family of tracial states. On the other hand, $C^{\ast}_{r}(\mathbb{F}_2)$ admits a faithful tracial state. Hence the maximal tensor product of two $C^{\ast}$-algebras admitting separating families of tracial states need not admit a separating family of tracial states. \\ \\
Another interesting usage of Proposition \ref{prop:max-min-faithful-tracial-state} is its relation to the Connes embedding problem, which is how it originally appears in \citep{Kirchberg-1993}. One equivalent formulation of the Connes embedding problem, see \citep[Theorem 13.3.1]{Brown-Ozawa-2008}, is that $C^{\ast}(\mathbb{F}_{\infty})\otimes_{\mathrm{max}} C^{\ast}(\mathbb{F}_{\infty}) = C^{\ast}(\mathbb{F}_{\infty})\otimes_{\mathrm{min}}C^{\ast}(\mathbb{F}_{\infty})$. If the maximal tensor product admitted a faithful tracial state, this equality would be true by Proposition \ref{prop:max-min-faithful-tracial-state}; see also \citep[Exercise 13.3.1-4]{Brown-Ozawa-2008}. By the announced negative answer to the Connes embedding problem \citep{MIP*=RE}, we would thus be able to conclude that $C^{\ast}(\mathbb{F}_{\infty}\times \mathbb{F}_{\infty})= C^{\ast}_{\mathbb{F}_{\infty}}\otimes_{\mathrm{max}}C^{\ast}_{\mathbb{F}_{\infty}}$ does not admit a faithful tracial state. \\ \\
Lastly, we examine some equivalent notions of having separating families of tracial states or admitting a faithful tracial state by using von Neumann terminology.
\begin{proposition}
Let $A$ be a unital $C^{\ast}$-algebra.
\begin{enumerate}
\item $A$ admits a separating family of tracial states if and only if $A$ unitally embeds into a finite von Neumann algebra.
\item $A$ has a faithful tracial state if and only if $A$ unitally embeds into a $\mathrm{II}_1$-factor.
\end{enumerate}
\end{proposition}
\begin{proof}
(i): Every finite von Neumann algebra has a separating family of tracial states by \citep[III.2.5.8]{Blackadar-Operator-Algebras-2006}, so we only need to prove the "only if" direction. For any tracial state $\tau\in T(A)$, we can consider the GNS-representation $\pi_{\tau}\colon A\to \mathbb{B}(H_{\tau})$. Since $\bigcap_{\tau \in T(A)}I_{\tau}=\{0\}$, it follows that the product $\pi\defeq \bigoplus_{\tau \in T(A)}\pi_{\tau}\colon A \to \mathbb{B}(H)$ is injective. Moreover,
\begin{align*}
A\cong \pi(A) \subseteq \bigoplus_{\tau\in T(A)}\pi_{\tau}(A)''
\end{align*}
and the right-hand side is a finite von Neumann algebra. \\ \\
(ii): A $\mathrm{II}_1$-factor immediately admits a faithful tracial state so to prove the other direction, let us assume that $A$ admits a faithful tracial state $\tau$. Then $A$ unitally embeds via the GNS representation of $\tau$ into a finite von Neumann algebra $M$ with a normal faithful tracial state. The claim now follows from the well-known fact that any von Neumann algebra with a normal faithful tracial state embeds into a $\mathrm{II}_1$-factor, see e.g.\ the proof of \citep[Theorem A.1]{Musat-Rordam-2020-non-closure-Ozawa-appendix}
\end{proof}
\section{Faithful tracial states on quotients}
In \citep{Murphy-2000}, Murphy initiated the study of $C^{\ast}$-algebras whose quotients all admit tracial states. He coined this notion QTS for \emph{quotient tracial states} and examples of $C^{\ast}$-algebras with the QTS property include unital strongly quasidiagonal $C^{\ast}$-algebras and group-$C^{\ast}$-algebras of amenable groups. In this section, we shall consider a stronger condition, namely that all quotients of the $C^{\ast}$-algebra admit \emph{faithful} tracial states. Continuing the terminology introduced by Murphy, we shall call this the \emph{QFTS property}. Let us look at a few examples of $C^{\ast}$-algebras with the QFTS property.
\begin{example}
The following $C^{\ast}$-algebras have the QFTS property:
\begin{enumerate}
\item Unital $C^{\ast}$-algebras with the QTS property and $T_1$ primitive ideal space,
\item $C^{\ast}(G)$ for virtually nilpotent groups $G$,
\item Subhomogeneous $C^{\ast}$-algebras.
\end{enumerate}
\end{example}
\begin{proof}
(i): Suppose $A$ is a unital $C^{\ast}$-algebra with the QTS property, and suppose that the primitive ideal space $\mathrm{Prim}(A)$ is $T_1$ or, equivalently, that all primitive quotients are simple. Since $A$ has the QTS property, all primitive quotients will therefore admit faithful tracial states. Now, if $I$ is any ideal in $A$, then $I$ is equal to the intersection of all primitive ideals containing it, and we get an embedding
\begin{align*}
A/I \hookrightarrow  \prod_{I\subseteq J \in \mathrm{Prim}(A)}A/J,
\end{align*}
and we have just proved that the right-hand side admits faithful tracial states. \\ \\
(ii): If $G$ is virtually nilpotent, then $C^{\ast}(G)$ has a $T_1$ primitive ideal space by \citep[Corollary 3.2]{Echterhoff-1990}, and moreover $C^{\ast}(G)$ has the QTS property. The result now follows from (i). \\ \\
(iii) Any quotient of a subhomogeneous $C^{\ast}$-algebra is again subhomogeneous, so it suffices to show that unital subhomogeneous $C^{\ast}$-algebras admit faithful tracial states, which follows from Example \ref{ex:separating-families-examples}.
\end{proof}
It is immediate that having the QFTS property implies admitting a faithful tracial state, and the converse fails in general as is shown in Example \ref{ex:faithful-tracial-not-passing-to-quotient}. One way of viewing the QFTS property is by the fact that the ideals are completely characterised by the tracial states. 
\begin{proposition}
Let $A$ be a separable, unital $C^{\ast}$-algebra. Then $A$ has the QFTS property if and only if all ideals of $A$ can be realised as the trace-kernel of a tracial state on $A$.
\end{proposition}
\begin{proof}
Suppose that $A$ has the QFTS property and let $I$ be an ideal in $A$. Denote by $\pi\colon A\to A/I$ the quotient mapping. By the QFTS property, $A/I$ admits a faithful tracial state $\tilde{\tau}$, which, in turn, induces a tracial state $\tau=\tilde{\tau}\circ \pi$ on $A$. It is clear that $I\subseteq I_{\tau}$, so suppose that $a\in I_{\tau}$. Then
\begin{align*}
0 = \tau(a^{\ast}a) = \tilde{\tau}(\pi(a^{\ast}a)) = \tilde{\tau}(\pi(a)^{\ast}\pi(a)),
\end{align*}
and by faithfulness of $\tilde{\tau}$, we see that $\pi(a)=0$ and, hence, that $a\in I$. \\ \\
Now suppose that, for any ideal $I$ in $A$, there exists a tracial state $\tau\in T(A)$ such that $I_{\tau}=I$. We can hence induce a tracial state $\tilde{\tau}$ on $A/I$ via $\tilde{\tau}(\pi(a))=\tau(a)$, where $\pi\colon A \to A/I_{\tau}$ is the quotient map. It is easily verified that $\tilde{\tau}$ is a faithful tracial state, which completes the proof.
\end{proof}
Using the equivalences of Proposition \ref{prop:faithful-trace-iff-separating-family} and the fact that stable $C^{\ast}$-algebras cannot admit bounded traces, we obtain the following.
\begin{proposition} \label{prop:QFTS-iff-intermediate-quotients-bdd-trace}
A separable, unital $C^{\ast}$-algebra $A$ has the QFTS property if and only if every intermediate quotient\footnote{Recall that an \emph{intermediate quotient} of $A$ is a $C^{\ast}$-algebra of the form $I/J$, where $J\subseteq I \subseteq A$ are ideals in $A$. In other words, any intermediate quotient can be realised as the ideal of a quotient of $A$.} admits a bounded trace. In particular, if $A$ has the QFTS property, then $A$ has no stable intermediate quotient and no properly infinite quotients.
\end{proposition}
One goal of this article is to provide a converse to the latter part of Proposition \ref{prop:QFTS-iff-intermediate-quotients-bdd-trace} and, consequence, give an equivalent reformulation of the QFTS property for some classes of $C^{\ast}$-algebras. We shall attack this problem in two different manners: One by using the connections between dimension functions on Cuntz semigroups and tracial states, and another by using a result on stability of hereditary $C^{\ast}$-subalgebras in \citep{Hirshberg-Rordam-Winter-2007}. \\ \\
Let $A$ be a separable, unital $C^{\ast}$-algebra. In the sequel, we will work with both the complete Cuntz semigroup $\mathrm{Cu}(A)$ and pre-complete Cuntz semigroup $\mathrm{W}(A)$. The reader may find information about these constructions in e.g.\ \citep{Antoine-Perera-Thiel-2019}. Our reason to be interested in these Cuntz semigroup is the intrinsic connection between dimension functions and (quasi)traces on $A$. On a general ordered Abelian semigroup $S$, a state $f$ is a linear functional preserving the order, and the collection of such is denoted by $\Sigma(S)$, and the collection of states normalised at some $t\in S$ is denoted by $\Sigma(S,t)$. By a \emph{dimension function} on $A$, we mean a state on $\mathrm{W}(A)$ normalized at $\braket{1_A}$, i.e., a state $d$ for which $d(\braket{1_A})=1$. For each $a\in M_{\infty}(A)_{+}$ and each $\varepsilon>0$, we define the $\varepsilon$-cutoff of $a$ by $(a-\varepsilon)_{+}=f_{\varepsilon}(a)$, where the right-hand side is the continuous functional calculus applied to the function $f_{\varepsilon}(t)=\max\{0,t-\varepsilon\}$. A dimension function $d$ is \emph{lower semi-continuous} if $d(\braket{(a-\varepsilon)_{+}})\to d(a)$ for all $a\in M_{\infty}(A)_{+}$ as $\varepsilon\to 0$. We denote by $\Sigma(\mathrm{W}(A),\braket{1_A})$ and $\Sigma_{\mathrm{lsc}}(\mathrm{W}(A),\braket{1_A})$ the collection of dimension functions and lower semi-continuous dimension functions on $A$, respectively. \\ \\
If $\tau$ is a tracial state (or, in general, a quasitrace) on $A$, one can associate to it a lower semi-continuous dimension function $d_{\tau}$ by $d_{\tau}(\braket{a}) = \lim_{n\to \infty}\tau(a^{1/n})$ for $a\in M_{\infty}(A)_{+}$ and, in fact, this association gives rise to an affine bijection between the collection of quasitraces $\mathrm{QT}(A)$ and lower semi-continuous dimension functions $\Sigma_{\mathrm{lsc}}(\mathrm{W}(A),\braket{1_A})$ \citep{Blackadar-Handelman-1982}. We can use this identification to give another characterisation when $C^{\ast}$-algebras have separating families of tracial states in terms of dimension functions.
\begin{proposition} \label{prop:sep-family-traces-iff-sep-family-dim-functions}
Let $A$ be a unital, exact $C^{\ast}$-algebra. Then $A$ has a separating family of tracial states if and only if $\mathrm{W}(A)$ has a separating family of lower semi-continuous dimension functions.
\end{proposition}
\begin{proof}
By a separating family of lower semi-continuous dimension functions we mean that, for any $a\in M_{\infty}(A)_{+}$, there exists $d\in \Sigma_{\mathrm{lsc}}(\mathrm{W}(A),\braket{1_A})$ for which $d(\braket{a})\neq 0$. It follows from \citep{Haagerup-Quasitraces-2014} and the aforementioned identification of quasitraces and dimension functions that we have an affine bijection between $T(A)$ and $\Sigma_{\mathrm{lsc}}(\mathrm{W}(A),\braket{1_A})$ via $\tau\mapsto d_{\tau}$. The statement of the proposition is merely rewriting the notion of separating family of tracial states via this association.
\end{proof}
Proposition \ref{prop:sep-family-traces-iff-sep-family-dim-functions} entails that understanding the structure of $\Sigma_{\mathrm{lsc}}(\mathrm{W}(A),\braket{1_A})$ will help in understanding when $C^{\ast}$-algebras admit faithful tracial states. The following proposition due to Goodearl and Handelman \citep[Proposition 4.2]{Goodearl-Handelman-1976} gives a characterisation of when certain states exists on ordered Abelian semigroups. For the statement, recall that an element $t$ on an ordered Abelian semigroup is \emph{properly infinite} if $2t\leq t$.
\begin{proposition}[Goodearl-Handelman, 1972]
Let $S$ be an ordered Abelian semigroup with a distinguished order unit $u$ and assume that no multiple of $u$ is properly infinite. Let $t\in S$ be arbitrary and set
\begin{align*}
\alpha_{\ast} &= \sup\{k/\ell \mid  k,\ell \in \NN \mbox{ and } ku\leq \ell t\}, \\
\alpha^{\ast} &= \inf\{k/\ell \mid  k,\ell \in \NN \mbox{ and } \ell t\leq ku\}.
\end{align*}
Then $0\leq \alpha_{\ast}\leq \alpha^{\ast}$ and there exists $d\in \Sigma(S,u)$ with $d(t)=\alpha$ if and only if $\alpha_{\ast}\leq \alpha \leq \alpha^{\ast}$.
\end{proposition}
From this we immediately get the following corollary.
\begin{corollary} \label{cor:not-separating-iff-inf-property}
Let $S$ is an ordered Abelian semigroup with a distinguished order unit, and let $t\in S$. Then $d(t)=0$ for all $d\in \Sigma(S,u)$  if and only if
\begin{align*}
\alpha^{\ast} = \inf\{k/\ell \mid k,\ell\in \NN \mbox{ and } \ell t \leq ku\} = 0.
\end{align*}
\end{corollary}
Observe that while this corollary seems almost directly applicable to Proposition \ref{prop:sep-family-traces-iff-sep-family-dim-functions}, there is an important subtlety, namely that in that proposition we are interested in separating families of \emph{lower semi-continuous} dimension functions. To remedy this, we need to examine more properties that Cuntz semigroups may have. If $S$ is an ordered Abelian semigroup and $x,y\in S$, we write that $x<_{s} y$ if there exists $n\in \NN$ such that $(n+1)x\leq ny$ or, equivalently, if there exists $m\in \NN$ such that $x\leq my$ and $d(x)<d(y)$ for all $d\in \Sigma(S,y)$ \citep{Ortega-Perera-Rordam-2012}.
\begin{definition}
Let $S$ be an ordered Abelian semigroup, and let $n\in \NN_0$. We say that $S$ has \emph{n-comparison} if, for any $x,y_0,\ldots, y_n\in S$ satisfying $x<_{s}y_j$ for all $j=0,\ldots,n$, we have $x\leq y_0+\ldots+y_n$.
\end{definition}
Clearly $n$-comparison implies $m$-comparison for $n\leq m$. Observe that $0$-comparison is equivalent to the property that $(n+1)x\leq ny$ for some $n \in \NN$ implies $x\leq y$, which is also known as \emph{almost unperforation}. If $A$ is a $C^{\ast}$-algebra, then $\mathrm{W}(A)$ is almost unperforated if and only if $\mathrm{Cu}(A)$ is almost unperforated \citep{Antoine-Perera-Thiel-2019}. Examples of $C^{\ast}$-algebras with almost unperforated Cuntz semigroups include $A\otimes B$, where $A$ is any unital $C^{\ast}$-algebra and $B$ is either an infinite-dimensional UHF-algebra \citep[Lemma 5.1]{Rordam-1992-II}, or $B=\mathcal{Z}$ is the Jiang-Su algebra \citep[Theorem 4.5]{Rordam-Stable-2004}. Another related, but weaker, notion is that of $\omega$-comparison; here we use the relation $\ll$ defined as follows: We say $s\ll t$ if, whenever $t\leq \sup_{n} t_n$ for an increasing sequence $(t_n)_{n\in \NN}$, there exists $n_0\in \NN$ such that $s\leq t_{n_0}$. 
\begin{definition}
A complete ordered Abelian semigroup $S$ has $\omega$-comparison if, for any $x,x',y_0,y_1,\ldots \in S$ with $x<_{s}y_j$ and $x'\ll x$, there exists $n\in \NN$ such that $x'\leq y_0+\ldots+y_n$.
\end{definition}
It is clear that $n$-comparison for any $n\in \NN$ implies $\omega$-comparison. \\ \\
We rewrite Corollary \ref{cor:not-separating-iff-inf-property} assuming almost unperforation.
\begin{lemma} \label{lemma:inf=0-iff-lt-leq-u}
Let $S$ be an ordered Abelian semigroup with a distinguished order unit $u$, and suppose that $S$ is almost unperforated. Let $t\in S$. The following are equivalent:
\begin{enumerate}
\item $d(t)=0$ for all $d\in \Sigma(S,u)$,
\item $\inf\{k/\ell \mid k,\ell \in \NN \mbox{ and } \ell t \leq ku\} = 0$,
\item There exists $k\in \NN$ such that $\ell t \leq ku$ for all $\ell \in \NN$,
\item $\ell t \leq u$ for all $\ell \in \NN$.
\end{enumerate}
\end{lemma}
\begin{proof}
It is clear that (iv)$\Rightarrow$(iii)$\Rightarrow$(ii) and (ii)$\Leftrightarrow$(i) is the content of Corollary \ref{cor:not-separating-iff-inf-property}, so let us prove (i)$\Rightarrow$(iv). Suppose $d(t)=0$ for all $d\in \Sigma(S,u)$. Since $d(\ell t)=0$ for all $\ell \in \NN$ and $d(u)=1$, whenever $d\in \Sigma(S,u)$, and as $u$ is an order unit for $S$, we find that $\ell t <_{s} u$, which implies $\ell t \leq u$ by almost unperforation of $S$. 
\end{proof}
Note that this proposition needs almost unperforation to be true. Consider, as constructed in \citep[Example 4.13]{Bosa-Petzka-2018}, the Cuntz semigroup $S=\{0, 1, \infty\}$ equipped with $1+1=\infty$ and the usual ordering. Observe that $S$ has $1$-comparison, but it is not almost unperforated. Note that $u=1$ is an order unit, and that $\ell t \leq 2u = \infty$ for all $t\in S$ and $\ell \in \NN$, such that (i) above holds. However, (iv) is false, since $\infty \not \leq 1$. \\ \\
An interesting related property to the equivalent conditions of Lemma \ref{lemma:inf=0-iff-lt-leq-u} is the notion of $\beta$-comparison introduced by Bosa-Petzka in \citep{Bosa-Petzka-2018}.
\begin{definition}
Let $S$ be an ordered Abelian semigroup. Define for each $x,y\in S$ the quantity
\begin{align*}
\beta(x,y) = \inf\{k/\ell \mid k,\ell \in \NN \mbox{ and } \ell x \leq ky\}.
\end{align*}
We say that $S$ has \emph{$\beta$-comparison} if, whenever $x,y\in S$ satisfies $\beta(x,y)=0$, then $x\leq y$.
\end{definition}
It holds in general that $\beta$-comparison implies $\omega$-comparison, but the above example shows that they are not equivalent properties for ordered Abelian semigroups, cf.\ \citep{Bosa-Petzka-2018}. For simple $C^{\ast}$-algebras, however, they are equivalent by the arguments in \citep[Section 5]{Bosa-Petzka-2018}. Note that almost unperforation implies $\beta$-comparison; the easiest way of seeing this is to note that $\beta(x,y)<1$ if and only if $x<_{s} y$.
\begin{lemma}
Let $S$ be an Abelian semigroup with a distinguished order unit $u$ and let $t\in S$ be arbitrary. Suppose that $S$ has $\beta$-comparison. Then (i)--(iv) in Lemma \ref{lemma:inf=0-iff-lt-leq-u} are equivalent.
\end{lemma}
\begin{proof}
We only need to prove (ii)$\Rightarrow$(iv). It is easily seen that, if $\beta(x,y)=0$, then $\beta(\ell x,y)=0$ for all $\ell\in \NN$. Observe that (ii) is equivalent to $\beta(t,u)=0$, and that this implies $\beta(\ell t, u)=0$ for all $\ell \in \NN$. Since $S$ has $\beta$-comparison, we conclude that $\ell t\leq u$ for any $\ell \in \NN$.
\end{proof}
The next proposition follows immediately from \citep[Lemma 2.4(ii)]{Robert-Rordam-2013}, and we extend the result in the subsequent lemma assuming stable rank one.
\begin{proposition} \label{prop:lt-leq-u-implies-finitely-many-mutually-orth}
Let $A$ be a unital $C^{\ast}$-algebra and assume that $\ell \braket{a}\leq \braket{1_A}$ for some $\ell \in \NN$ and $a\in A$. Then, for any $\varepsilon>0$, there exists mutually orthogonal, mutually equivalent positive elements $e_1,\ldots, e_{\ell}\in A$ such that $e_i \sim (a-\varepsilon)_{+}$ for all $i=1,\ldots, \ell$.
\end{proposition}
\begin{lemma} \label{lemma:lt-leq-u-implies-mutually-orth-seq}
Let $A$ be a unital $C^{\ast}$-algebra with stable rank one and assume that there exists $a\in A$ such that $\ell \braket{a}\leq \braket{1_A}$ for all $\ell \in \NN$. Then, for any $\varepsilon>0$, there exists a sequence mutually orthogonal, mutually equivalent positive elements $(e_{\ell})_{\ell \in \NN}\subseteq A$ such that $e_{\ell}\sim (a-\varepsilon)_{+}$ for all $\ell \in \NN$.
\end{lemma}
\begin{proof}
Suppose that $\ell\braket{a}\leq \braket{1_A}$ for all $\ell \in \NN$. Let $n\in \NN$ be arbitrary, then we can use Proposition \ref{prop:lt-leq-u-implies-finitely-many-mutually-orth} to construct $n$ pairwise orthogonal and pairwise equivalent positive elements $e_1,\ldots, e_n$ such that $e_i\sim (a-\varepsilon)_{+}$ for some $\varepsilon>0$. Let $b=\sum_{i=1}^{n}e_i$ and define for each $\delta>0$ the function
\begin{align*}
h_{\delta}(t) = \begin{cases}
\frac{\delta-t}{\delta} &\mbox{ if } 0\leq t \leq \delta \\
0 &\mbox{ if } t\geq \delta
\end{cases}.
\end{align*}
Observe that $(b-\delta)_{+}\perp h_{\delta}(b)$ and that $b+h_{\delta}(b)$ is invertible. Let $\delta_{n}>0$ be arbitrary, then we obtain the following inequalities:
\begin{align*}
n\braket{(a-\varepsilon)_{+}}+\braket{h_{\delta_{n}}(b)} \geq \braket{b}+\braket{h_{\delta_{n}}(b)}\geq \braket{u} \geq (n+1)\braket{a} = n\braket{a} + \braket{a}.
\end{align*}
By the cancellation theorem of \citep[Theorem 4.3]{Rordam-Winter-2010}, we find that $\braket{a}\leq \braket{h_{\delta_{n}}(b)}$. Using \citep[Lemma 2.4]{Robert-Rordam-2013}, we construct $c\in \overline{h_{\delta_{n}}(b)Ah_{\delta_{n}}(b)}$ such that $c\sim (a-(\varepsilon+\delta_{n}))_{+}$. In particular, we find that $c\perp (e_i-\delta_{n})_{+}$ for all $i=1,\ldots, n$, and that, for any $i=1,\ldots, n$,
\begin{align*}
c\sim (a-(\varepsilon+\delta_{n}))_{+} = ((a-\varepsilon)_{+}-\delta_{n})_{+} \sim (e_i-\delta_{n})_{+}.
\end{align*}
We have hence shown that, if we can construct $n$ pairwise orthogonal and pairwise equivalent positive elements $e_1,\ldots, e_n$, then we can, for any $\delta>0$, construct a positive element $e_{n+1}\in A$ such that $e_{n+1}\perp (e_{i}-\delta)_{+}$ and $e_{n+1}\sim (e_{i}-\delta)_{+}$ for any $i=1,\ldots, n$. Now we use this process inductively. Start with any positive element $e_1\in A$ and let $\delta_{1}>0$. Use the above method to construct an element $e_{2}\in A$ with the properties that $e_{2}\perp (e_{1}-\delta_{1})_{+}$ and $e_{2}\sim (e_{1}-\delta_{1})_{+}$. Now let $\delta_{2}>\delta_{1}$ and construct $e_{3}\in A$ such that $e_{3}\sim (e_{2}-\delta_{2})_{+}\sim (e_{1}-(\delta_{1}+\delta_{2}))_{+}$ and that these three elements are mutually orthogonal. Continuing this process, we can thus for any increasing sequence $0<\delta_{1}<\delta_{2}<\cdots$ construct a sequence of positive elements $(e_{n})_{n\in \NN}\subseteq A$ with the property that, for any $m>n$, 
\begin{align*}
e_{m} \sim (e_{n}-\sum_{k=n}^{m-1}\delta_{k})_{+}
\end{align*}
and that these are orthogonal. Since we have no restrictions on the choice of the sequence $(\delta_{n})_{n\in \NN}$, we may take it to be any sequence such that the series $\delta=\sum_{n=1}^{\infty}\delta_{n}$ converges with limit $\delta < \norm{e_1}$. Define $e_{n}' = (e_{n}-\sum_{k=n}^{\infty}\delta_{k}))_{+}$ for each $n\in \NN$, then we see for any $m>n$ that
\begin{align*}
e_{n}' = (e_{n}-\sum_{k=n}^{\infty}\delta_{k})_{+} \sim (e_{n}-(\sum_{k=n}^{m-1}\delta_{k}+\sum_{k=m}^{\infty}\delta_{k}))_{+} \sim (e_{m}-\sum_{k=m}^{\infty}\delta_{k})_{+} = e_{m}'.
\end{align*}
The proof of mutual orthogonality goes in a similar fashion. We have thus proved the existence of a sequence of pairwise orthogonal and pairwise equivalent elements in $A$.
\end{proof}
\begin{theorem} \label{thm:no-faithful-trace-iff-no-stable-subalgebra}
Let $A$ be a separable, unital, exact $C^{\ast}$-algebra with stable rank one. Assume that $\mathrm{W}(A)$ is almost unperforated. The following conditions are equivalent:
\begin{enumerate}
\item $A$ admits no faithful tracial states,
\item $A$ contains a non-zero stable $C^{\ast}$-subalgebra,
\item $A$ contains a non-zero stable hereditary $C^{\ast}$-subalgebra.
\end{enumerate}
\end{theorem}
\begin{proof}
The equivalence (ii)$\Leftrightarrow$(iii) holds by \citep{Hjelmborg-Rordam-1998}. We saw that (ii)$\Rightarrow$(i) in Proposition \ref{prop:faithful-trace-implies-stably-finite-and-no-stable}, so we only need to show (i)$\Rightarrow$(ii). By Proposition \ref{prop:sep-family-traces-iff-sep-family-dim-functions}, there exists some $\braket{a}\in \mathrm{W}(A)$ such that $d(\braket{a})=0$ for all $d\in \Sigma_{\mathrm{lsc}}(\mathrm{W}(A),\braket{1_A})$. Since $\mathrm{W}(A)$ is assumed to be almost unperforated, the set $\Sigma_{\mathrm{lsc}}(\mathrm{W}(A),\braket{1_A})$ is dense in $\Sigma(\mathrm{W}(A),\braket{1_A})$ in the topology of pointwise convergence \citep[Theorem 3.3]{de-Silva-2016}, and hence $d(\braket{a})=0$ for all $d\in \Sigma(\mathrm{W}(A),\braket{1_A})$. Combining Corollary \ref{cor:not-separating-iff-inf-property} and Lemma \ref{lemma:inf=0-iff-lt-leq-u} gives us that $\ell \braket{a}\leq \braket{1_A}$ for all $\ell \in \NN$. By Lemma \ref{lemma:lt-leq-u-implies-mutually-orth-seq}, this implies the existence of a sequence of mutually orthogonal, mutually equivalent positive elements $(e_\ell)_{\ell \in \NN}$ in $A$. It then follows from \citep{Hjelmborg-Rordam-1998} that the hereditary $C^{\ast}$-algebra generated by $\{e_{\ell}\}_{\ell \in \NN}$ is a stable $C^{\ast}$-subalgebra of $A$.
\end{proof}
Note that Theorem \ref{thm:no-faithful-trace-iff-no-stable-subalgebra} provides a converse result to Proposition \ref{prop:faithful-trace-implies-stably-finite-and-no-stable}. It is, hence, of independent interest, but one can also use the theorem to give an equivalent formulation of QFTS. We shall, however, use a different method, with which we can avoid the assumption of stable rank one; do note, however, that this assumption would make the assumption of no properly infinite unital quotients vacuously true in the results to come. \\ \\
We shall in the sequel use the following result due to Hirshberg-Rørdam-Winter, \citep[Theorem 3.6]{Hirshberg-Rordam-Winter-2007}, which here includes a slight modification due to Haagerup in that all quasitraces on exact $C^{\ast}$-algebras are tracial states.
\begin{theorem}[Hirshberg-Rørdam-Winter, 2007] \label{thm:hirchberg-rordam-winter-stable-hereditary}
Let $A$ be a separable, unital, exact $C^{\ast}$-algebra for which $\mathrm{Cu}(A)$ has $\omega$-comparison. Let $B\subseteq A$ be a hereditary $C^{\ast}$-subalgebra. Then $B$ is stable if and only if $B$ admits no non-zero tracial states and no quotient of $B$ is unital.
\end{theorem}
The next result is well-known, but we supply a proof for the sake of completeness.
\begin{proposition} \label{prop:regularity-of-prop-inf}
Let $A$ be a unital $C^{\ast}$-algebra such that $\mathrm{W}(A)$ is almost unperforated. If $M_n(A)$ is properly infinite for some $n\in \NN$, then $A$ is properly infinite.
\end{proposition}
\begin{proof}
As $M_n(A)$ is properly infinite, we have that $(n+1)\braket{1_A}\leq n\braket{1_A}$. By induction , $(n+k)\braket{1_A}\leq n\braket{1_A}$ for all $k\in \NN$, and putting $k=n+2$ gives us $2(n+1)\braket{1_A}\leq n\braket{1_A}$. Almost unperforation of $\mathrm{W}(A)$ implies that $2\braket{1_A}=\braket{1_A}$, i.e., $A$ is properly infinite.
\end{proof}
\begin{proposition} \label{prop:no-faithful-trace}
Let $A$ be a separable, unital, exact $C^{\ast}$-algebra such that $\mathrm{Cu}(A)$ is almost unperforated. If $A$ has no faithful tracial state, then either $A$ has a stable ideal or a unital, properly infinite quotient (or both).
\end{proposition}
\begin{proof}
Suppose that $A$ has no faithful tracial state. By Proposition \ref{prop:faithful-trace-iff-separating-family} there exists an ideal $I$ in $A$ with no tracial state. Note that almost unperforation passes to ideals and quotients. Suppose that $I$ is not stable, then Theorem \ref{thm:hirchberg-rordam-winter-stable-hereditary} implies that $I$ has a unital quotient $I/J$. Observe that if $I/J$ had a tracial state, then so would $I$ via the quotient map $I\to I/J$. Since $I/J$ has no tracial state there exists $n\in \NN$ such that $M_n(I/J)$ is properly infinite \citep{Haagerup-Quasitraces-2014}. It follows from Proposition \ref{prop:regularity-of-prop-inf} that $I/J$ is properly infinite. Now observe that we, from unitality of $I/J$, obtain an isomorphism $A/J \cong A/I \oplus I/J$, and hence $I/J$ is a unital and properly infinite quotient of $A$.
\end{proof}
We cannot immediately turn Proposition \ref{prop:no-faithful-trace} into an equivalent reformulation of admitting a faithful tracial state, since admitting a faithful tracial state does not inhibit the existence of properly infinite quotients; for example, the full group-$C^{\ast}$-algebra $C^{\ast}(\mathbb{F}_2)$ has a faithful tracial state as it is residually finite-dimensional, but any separable $C^{\ast}$-algebra can be realised as a quotient of it. One way of salvaging this is to assume some more properties on $A$, e.g., QTS or stable rank one, which would give a necessary condition for having a faithful tracial state.
\begin{corollary}
Suppose $A$ is a separable, unital, exact $C^{\ast}$-algebra such that $\mathrm{Cu}(A)$ is almost unperforated. Assume that either $A$ has the QTS property, or that $A$ has stable rank one. Then $A$ has a faithful tracial state if and only if $A$ has no stable ideal.
\end{corollary}
\begin{proof}
Apply Proposition \ref{prop:no-faithful-trace} and note that $A$ cannot have any properly infinite quotients, since it either has QTS or stable rank one.
\end{proof}
Since almost unperforation of Cuntz semigroups is easily seen to pass to quotients, we can also apply Proposition \ref{prop:no-faithful-trace} to the quotients and obtain the following result.
\begin{theorem} \label{thm:QFTS-iff-no-prop-inf-and-no-stable}
Let $A$ be a separable, unital, exact $C^{\ast}$-algebra satisfying that $\mathrm{Cu}(A)$ is almost unperforated. Then $A$ has the QFTS property if and only if $A$ has no stable intermediate quotients and no unital, properly infinite quotients.
\end{theorem}
We now look at some ways of how one may apply the above theorem. Firstly, as mentioned previously, by assuming the QTS property one may disregard the possibility of unital, properly infinite quotients; since we are interested in the existence of \emph{faithful} tracial states on the quotients, this is clearly not a big assumption. There are several classes of $C^{\ast}$-algebras with the QTS property such as the class of exact, unital $C^{\ast}$-algebras with stable rank one, and the class of group-$C^{\ast}$-algebras of amenable groups. In general, unital, nuclear $C^{\ast}$-algebras have the QTS property if and only if they are hypertracial, cf.\ \citep{Bedos-1995-hypertraces}. Assuming almost unperforation on the level of Cuntz semigroups might seem like a strong assumption, but we can luckily invoke various results to weaken this to just $\omega$-comparison on all quotients which, by Robert \citep{Robert-Munster-2011}, is provided given finite nuclear dimension.
\begin{corollary} \label{cor:fin-nuc-QTS-then-QFTS-iff-no-SIQ}
Let $A$ be a separable, unital $C^{\ast}$-algebra with finite nuclear dimension and QTS. Then $A$ has the QFTS property if and only if $A$ has no stable intermediate quotient.
\end{corollary}
\begin{proof}
One direction is immediate, so assume that $A/I$ is a quotient of $A$ without a faithful tracial state. Since $A/I$ is unital, it holds by Proposition \ref{prop:permanence-properties-sep-family} that $A/I\otimes_{\mathrm{min}} \mathcal{Z}$ has no faithful tracial state. Since $A/I\otimes_{\mathrm{min}} \mathcal{Z}$ has QTS and an almost unperforated Cuntz semigroup, it follows from Proposition \ref{prop:no-faithful-trace} that $A/I\otimes_{\mathrm{min}} \mathcal{Z}$ has a stable ideal. We thus find an ideal $J$ in $A/I$ (i.e., an intermediate quotient of $A$) such that $J\otimes_{\mathrm{min}} \mathcal{Z}$ is stable. Since the Cuntz semigroup of $J\otimes_{\mathrm{min}} \mathcal{Z}$ is almost unperforated, it follows from \citep{Hirshberg-Rordam-Winter-2007} that $J\otimes_{\mathrm{min}} \mathcal{Z}$ has no bounded traces and no unital quotients. In particular, neither does $J$, since any such would be easily extended from $J$ to $J\otimes_{\mathrm{min}} \mathcal{Z}$. Since $A$ has finite nuclear dimension, which is preserved by taking ideals and quotients, $J$ has finite nuclear dimension and thus $\omega$-comparison by \citep[Theorem 1]{Robert-Munster-2011}. A direct application of \citep[Proposition 4.7]{Ortega-Perera-Rordam-2011} then proves that $J$ is stable, which finalises the proof.
\end{proof}
It is not immediate how one can use Corollary \ref{cor:fin-nuc-QTS-then-QFTS-iff-no-SIQ} to prove that specific $C^{\ast}$-algebras have the QFTS property, since proving the lack of stable intermediate quotients is a difficult task. However, we can use the result to prove a dichotomy for certain classes of $C^{\ast}$-algebras.
\begin{example}
Let $X$ be a compact Hausdorff space with finite covering dimension and consider an action $\alpha$ of $\ZZ$ on $C(X)$. Since $C(X)$ is Abelian, it follows from \citep[Proposition 3.7]{Bedos-1995-hypertraces} that the crossed product $C(X)\rtimes_{\alpha} \ZZ$ has the QTS property. Moreover, by \citep{Hirshberg-Wu-2017}, $C(X)\rtimes_{\alpha} \ZZ$ has finite nuclear dimension. It thus follows from Corollary \ref{cor:fin-nuc-QTS-then-QFTS-iff-no-SIQ} that $C(X)\rtimes_{\alpha}\ZZ$ has the QFTS property if and only if it has no stable intermediate quotients. In particular, this result holds for groups $G\rtimes_{\alpha} \ZZ$ whenever $G$ is an Abelian group with finite-dimensional Pontryagin dual.
\end{example}
Recall that if $A$ is a unital $C^{\ast}$-algebra with a non-unitary isometry $v$, then we can embed the Toeplitz algebra $\mathcal{T}=C^{\ast}(v)$ in $A$, and we can identify $\mathbb{K}$ as an ideal in $\mathcal{T}$. Therefore, any infinite, unital $C^{\ast}$-algebra contains a stable $C^{\ast}$-subalgebra. By Corollary \ref{cor:fin-nuc-QTS-then-QFTS-iff-no-SIQ}, we may extend this to prove the existence of stable \emph{ideals} of certain $C^{\ast}$-algebras and their quotients.
\begin{example}
Consider the Lamplighter group $G=\ZZ/\ZZ_2 \wr \ZZ$, which is an amenable group and whose $C^{\ast}$-algebra has finite nuclear dimension by \citep{Hirshberg-Wu-2017}. The $C^{\ast}$-algebra $C^{\ast}(G)$ admits an infinite quotient \citep[Corollary 3.5]{Carrion-Dadarlat-Eckhardt-Quasidiagonal-2013} such that is does not have the QFTS property, and Corollary \ref{cor:fin-nuc-QTS-then-QFTS-iff-no-SIQ} then implies the existence of a stable intermediate quotient of $C^{\ast}(G)$.
\end{example}
By the exact same line of reasoning as in Corollary \ref{cor:fin-nuc-QTS-then-QFTS-iff-no-SIQ}, we obtain the following equivalent reformulation of when certain $C^{\ast}$-algebras have faithful tracial states.
\begin{corollary} \label{cor:faithful-trace-iff-no-stable-ideals}
Let $A$ be a separable, unital $C^{\ast}$-algebra. Assume that $\mathrm{Cu}(A)$ has $\omega$-comparison (e.g., $A$ has finite nuclear dimension), and that $A$ has no properly infinite quotients (e.g., $A$ has the QTS property or stable rank one). Then $A$ has a faithful tracial state if and only if $A$ has no stable ideals.
\end{corollary}
Blackadar showed \citep{Blackadar-Traces-on-AF-1980} that an AF-algebra is stable if and only if no ideals admit a bounded trace. Since being an AF-algebra is preserved by taking ideals, it thus follows from Proposition \ref{prop:faithful-trace-iff-separating-family} that an AF-algebra has a faithful tracial state if and only if it admits no stable ideals. As AF-algebras have finite nuclear dimension (in fact, a $C^{\ast}$-algebra has zero nuclear dimension if and only if it is an AF-algebra \citep{Winter-Zacharias-Nuclear-Dim-2010}) and stable rank one, Corollary \ref{cor:faithful-trace-iff-no-stable-ideals} can be seen as a generalisation of this result. \\ \\
Another use of the QFTS property is to prove strong quasidiagonality of $C^{\ast}$-algebras under the assumption that the UCT problem is true, as we shall demonstrate in this section. Recall that a *-representation $\pi\colon A\to \mathbb{B}(H)$ of a $C^{\ast}$-algebra $A$ is \emph{quasidiagonal} if there exists a net of finite rank projections $P_i$ in $\mathbb{B}(H)$, which strongly converges to the identity and for which $[\pi(a),P_i]\to 0$ for all $a\in A$. A $C^{\ast}$-algebra is then said to be \emph{quasidiagonal} if it admits a faithful quasidiagonal *-representation, and it is \emph{strongly quaisidiagonal} if all *-representations are quasidiagonal. Any quotient of a strongly quasidiagonal $C^{\ast}$-algebra is trivially (strongly) quasidiagonal, but a $C^{\ast}$-algebra whose quotients are all quasidiagonal need not be strongly quasidiagonal \citep[Example 20]{Brown-Dadarlat-1996}. The obstruction is the possible existence of a quotient with a faithful *-representation, which intersects the compacts, and hence the *-representation need not be quasidiagonal by \citep[Theorem 7.2.5]{Brown-Ozawa-2008}. By assuming the QFTS property, we can avoid this issue, and the difficulties now revolve around when $C^{\ast}$-algebras have quasidiagonal quotients. One way of achieving this is by the following, which also uses the tracial states, but with the caveat that the $C^{\ast}$-algebras need to be residually UCT.
\begin{theorem} \label{thm:sep-nuc-QD-faithful-trace-implies-SQD}
Let $A$ be a separable, nuclear, quasidiagonal $C^{\ast}$-algebra. Suppose that $A$ has the QFTS property, and that all quotients of $A$ satisfy the UCT. Then $A$ is strongly quasidiagonal.
\end{theorem}
\begin{proof}
It suffices to show that all irreducible *-representations of $A$ are quasidiagonal Let $\pi\colon A\to \mathbb{B}(H)$ be an irreducible *-representation of $A$ with kernel $I$, and let $\tilde{\pi}\colon A/I \to \mathbb{B}(H)$ be the corresponding faithful *-representation on the quotient. Obviously, $\pi$ and $\tilde{\pi}$ have the same image, and, by assumption, there exists a faithful tracial state on $A/I$, hence also on $\tilde{\pi}(A/I)$. Suppose that $\tilde{\pi}$ is not essential, that is, suppose that $\mathbb{K}(H)\cap \tilde{\pi}(A/I)\neq \{0\}$. As $\tilde{\pi}$ is an irreducible *-representation, this implies that $\mathbb{K}(H)$ is an ideal in $\tilde{\pi}(A/I)$. However, this is impossible, as we would then obtain a bounded trace on $\mathbb{K}(H)$. Therefore, $\tilde{\pi}$ is a faithful and essential *-representation. Since $A$ is assumed to have the QFTS property, there exists a faithful tracial state $\tau$ on $A/I$, which is quasidiagonal by \citep[Theorem 4.1]{Gabe-JoFA-2017}. Therefore, $\tilde{\tau}$ is a faithful quasidiagonal tracial state on $A/I$, and $A/I$ is therefore a quasidiagonal $C^{\ast}$-algebra by \citep[Proposition 1.4]{Tikuisis-White-Winter-Ann-2017}. Since $\tilde{\pi}$ is a faithful, essential *-representation of a quasidiagonal $C^{\ast}$-algebra, it is a quasidiagonal *-representation. Consequently, $\pi$ is a quasidiagonal *-representation, and as $\pi$ was an arbitrary irreducible *-representation of $A$, we conclude that $A$ is strongly quasidiagonal, which completes the proof.
\end{proof}
The assumption of all quotients satisfying the UCT is a strong assumption, and even for relatively well-known classes of $C^{\ast}$-algebras it is unclear when this occurs. In \citep{Eckhardt-Gillaspy-2016} it is proven that it holds true for (primitive) quotients of $C^{\ast}(G)$, whenever $G$ is finitely generated and nilpotent. However, it is seemingly very difficult to extend these results, and the authors of \citep{Eckhardt-Gillaspy-McKenney-2019} argue that resolving it for virtually nilpotent groups might not be any easier than resolving the UCT-problem for nuclear $C^{\ast}$-algebras. If all nuclear $C^{\ast}$-algebras were shown to be in the UCT-class, then the assumption of all quotients satisfying the UCT would, of course, be vacuous in Theorem \ref{thm:sep-nuc-QD-faithful-trace-implies-SQD} for nuclear $C^{\ast}$-algebras, but \citep[Proposition 5.1]{Gabe-JoFA-2017} would also resolve this. In particular, we can for group-$C^{\ast}$-algebras achieve the following, which modulo the UCT assumption includes all previously known examples of strongly quasidiagonal groups.
\begin{proposition}
Let $G$ be a countable, amenable group for which $C^{\ast}(G)$ has the QFTS property. If the UCT conjecture holds true, then $C^{\ast}(G)$ is strongly quasidiagonal.
\end{proposition}
Using the result on QFTS $C^{\ast}$-algebra from Corollary \ref{cor:fin-nuc-QTS-then-QFTS-iff-no-SIQ}, we get the following result.
\begin{corollary} \label{cor:no-stable-intermediate-implies-SQD}
Let $A$ be a separable, unital, quasidiagonal, QTS $C^{\ast}$-algebra with finite nuclear dimension, and assume all quotients of $A$ satisfy the UCT. If $A$ has no stable intermediate quotients, then $A$ is strongly quasidiagonal.
\end{corollary}
Proving that a given $C^{\ast}$-algebra is not strongly quasidiagonal often revolves around constructing a quotient with an infinite projection, see e.g.\ \citep[Section 3]{Carrion-Dadarlat-Eckhardt-Quasidiagonal-2013}. Since the existence of an infinite projection in the quotient would imply the existence of a stable $C^{\ast}$-\emph{subalgebra} on the quotient, one can hence view Corollary \ref{cor:no-stable-intermediate-implies-SQD} as a partial converse result to this method.
\section{Just-infinity and quasidiagonality}
In \citep{Grigorchuk-Musat-Rordam-2018}, Grigorchuk, Musat and Rørdam introduced the notion of just-infiniteness for $C^{\ast}$-algebra analogously to the group theoretical property of the same name. We say that a $C^{\ast}$-algebra $A$ is \emph{just-infinite} if it is infinite-dimensional and all its proper quotients are finite-dimensional. While these have a certain almost finite-dimensional flavour, they can admit quite exotic behaviour; for example, any infinite-dimensional and simple $C^{\ast}$-algebra is trivially just-infinite, e.g., the Cuntz algebra $\mathcal{O}_2$. Li initiated in \citep{Li-2019-Arxiv} the analysis of quasidiagonality of just-infinite $C^{\ast}$-algebras and proved, among other things, that quasidiagonality and inner quasidiagonality coincide among separable just-infinite $C^{\ast}$-algebra. This proof uses the classification of just-infinite $C^{\ast}$-algebras of \citep[Theorem 3.10]{Grigorchuk-Musat-Rordam-2018}, but there is an even stronger and more elementary proof without resorting to this classification as seen below.
\begin{theorem} \label{thm:just-infinite-and-qd-implies-sqd}
A just-infinite $C^{\ast}$-algebra is quasidiagonal if and only if it is strongly quasidiagonal.
\end{theorem}
\begin{proof}
It is clear that strongly quasidiagonal $C^{\ast}$-algebras are quasidiagonal, so suppose that $A$ is a quasidiagonal $C^{\ast}$-algebra. Observe that all non-faithful *-representations of $A$ are quasidiagonal since all the proper quotients of $A$ are finite-dimensional. Hence we only need to prove quasidiagonality of faithful irreducible *-representations. So assume that $\pi\colon A \to \mathbb{B}(H)$ is a faithful, irreducible *-representation, then there are two possibilities: Either $\pi$ is essential, or $\mathbb{K}(H)$ is an ideal in $\pi(A)$. If $\pi$ is essential, then as $A$ is quasidiagonal and $\pi$ is faithful, $\pi$ is quasidiagonal. So suppose that $\mathbb{K}(H)$ is an ideal in $\pi(A)$. Then either $\pi(A)\cong \mathbb{K}(H)$, or $\pi(A)$ is an extension of the AF-algebra $\mathbb{K}(H)$ and the finite-dimensional $C^{\ast}$-algebra $\pi(A)/\mathbb{K}(H)$. Consequently, $\pi(A)$ is an AF-algebra and, thus, strongly quasidiagonal. Hence $A$ is strongly quasidiagonal (and, in fact, an AF-algebra).
\end{proof}
\begin{corollary}
Let $G$ be a group such that $C^{\ast}(G)$ is just-infinite. Then $C^{\ast}(G)$ is strongly quasidiagonal.
\end{corollary}
\begin{proof}
Consider the canonical surjection $C^{\ast}(G)\to C^{\ast}_{r}(G)$. As $C^{\ast}(G)$ is just-infinite, we must either have that this is an isomorphism or that $C^{\ast}_{r}(G)$ is finite-dimensional (in which case the surjection is also an isomorphism). In particular, $G$ is amenable and hence quasidiagonal by \citep[Corollary C]{Tikuisis-White-Winter-Ann-2017}. Therefore, by Theorem \ref{thm:just-infinite-and-qd-implies-sqd}, $C^{\ast}(G)$ is strongly quasidiagonal.
\end{proof}
Following the terminology of \citep[Theorem 3.10]{Grigorchuk-Musat-Rordam-2018}, any just-infinite $C^{\ast}$-algebra of type ($\gamma$) is residually-finite dimensional and, hence, strongly quasidiagonal by Theorem \ref{thm:just-infinite-and-qd-implies-sqd}. Since type ($\alpha$) consists of all simple and infinite-dimensional $C^{\ast}$-algebras, there are plenty of quasidiagonal and non-quasidiagonal examples in this type. It is unknown which just-infinite $C^{\ast}$-algebras of type ($\beta$) are quasidiagonal, but the following trivial observation, which is contained in the proof of Theorem \ref{thm:just-infinite-and-qd-implies-sqd}, provides a subclass of quasidiagonal just-infinite $C^{\ast}$-algebras.
\begin{proposition}
If $A$ is a non-simple just-infinite $C^{\ast}$-algebra which has a non-essential irreducible *-representation, then $A$ is of type ($\beta$) and (strongly) quasidiagonal.
\end{proposition}
It is worth noting that, since separable just-infinite $C^{\ast}$-algebras are primitive \citep[Lemma 3.2]{Grigorchuk-Musat-Rordam-2018}, the existence of a faithful irreducible *-representation is guaranteed. Hence one only needs to examine this representation to determine strong quasidiagonality of just-infinite $C^{\ast}$-algebras: If it is non-essential \emph{or} it is essential \emph{and} quasidiagonal, the $C^{\ast}$-algebra is strongly quasidiagonal. We can summarise this in the following corollary, which gives the quasidiagonality dichotomy of just-infinite $C^{\ast}$-algebras.
\begin{corollary}
Let $A$ be a separable just-infinite $C^{\ast}$-algebra with a faithful irreducible *-representation $\pi\colon A\to \mathbb{B}(H)$. Then $A$ is strongly quasidiagonal if and only if $\pi$ is quasidiagonal.
\end{corollary}
\renewcommand{\bibname}{References}
\bibliography{qfts-bibtex} 
\bibliographystyle{abbrv}

\vspace{1cm}

\noindent Henning Olai Milhøj \\
Department of Mathematical Sciences\\
University of Copenhagen\\ 
Universitetsparken 5, DK-2100, Copenhagen \O\\
Denmark \\
milhoj@math.ku.dk\\

\end{document}